\documentclass[12pt,a4paper]{amsart}

\newtheorem{theo+}              {Theorem}           [section]
\newtheorem{prop+}  [theo+]     {Proposition}
\newtheorem{coro+}  [theo+]     {Corollary}
\newtheorem{lemm+}  [theo+]     {Lemma}
\newtheorem{exam+}  [theo+]     {Example}
\newtheorem{rema+}  [theo+]     {Remark}
\newtheorem{defi+}  [theo+]     {Definition}

\newenvironment{theorem}{\begin{theo+}}{\end{theo+}}
\newenvironment{proposition}{\begin{prop+}}{\end{prop+}}
\newenvironment{corollary}{\begin{coro+}}{\end{coro+}}

\usepackage{amsthm}
\theoremstyle{plain} \theoremstyle{remark}
\newtheorem{remark}{Remark}
\newtheorem{example}{Example}

\def\E{/\kern-1.0em \equiv }

\author{}

\begin{document}
\title[Equivariant and stable biharmonic maps]{A note on equivariant biharmonic maps and stable biharmonic maps}
\subjclass{58E20} \keywords{Equivariant biharmonic maps, rotationally symmetric biharmonic maps, Second variation formula of biharmonic maps, stable biharmonic maps.}
\author{Ye-Lin Ou $^{*}$}
\thanks{$^{*}$ This work was supported by a grant from the Simons Foundation ($\#427231$, Ye-Lin Ou).}
\address{Department of
Mathematics,\newline\indent Texas A $\&$ M University-Commerce,
\newline\indent Commerce, TX 75429,\newline\indent USA.\newline\indent
E-mail:yelin.ou@tamuc.edu }
\date{10/06/2019}
\maketitle
\section*{Abstract}
\begin{quote}  In this note, we generalize biharmonic equation for rotationally symmetric maps (\cite{BMO07a}, \cite{WOY14}, \cite{MOR15}) to equivariant maps between model spaces and use it to give a complete classification of rotationally symmetric conformal biharmonic maps from a $4$-dimensional space form into a $4$-dimensional model space. We also give an improved second variation formula for biharmonic maps into a space form and use it to prove that there exists no stable  proper biharmonic maps with constant square norm of tension field from a compact Riemannian manifold without boundary into a space form of positive sectional curvature.
{\footnotesize } 
\end{quote}
\section{Equivariant biharmonic maps and a classification of rotationally symmetric conformal biharmonic maps}
The warped product manifold of the type $I\times_fS^{m-1}$, where $f:I\to (0,\infty)$ is a real-valued function from an interval $I\subset \mathbb{R}$,  plays a very important role in mathematics and physics. Each of the three space forms of complete simply connected Riemannian manifolds of constant sectional curvature $c>0, c=0, \;{\rm or}\; c<0$ is locally a warped product of this type. In physics, Lorentz warped products of this type provide many exact solutions to the  Einstein field equations or the modified field equations, including the Schwarzschild solution and Robertson-Walker models (see e.g., \cite{Ch17} for more details).  Finally, the model spaces used by Greene and Wu \cite{GW79} to develop a good function  theory on complete simply connected Riemannian manifolds of nonpositive sectional curvature (i.e., Cartan-Hadamard manifolds) are also warped product of this type.

Recall that a {\it plole} of a manifold is a point $o\in M$ where the exponential map $exp_o \, : \, T_oM \rightarrow M$ is a diffeomorphism. A Riemannian manifold $(M^m(o),\,g)$ with a pole $o$ is called a \emph{model} if  every linear isometry of $T_oM$ is the differential at $o$ of an isometry of $M$.  A special feature of a model  lies in the fact that it can be described, by using geodesic polar coordinates centered at the pole $o$, as a warped product:
\begin{equation}\notag
(M^m(o),\,g) = \left (  [0,+\infty)\times S^{m-1},  dr^2 + f^2(r)\, g^{S^{m-1}}\right ),
\end{equation}
where $ (\,S^{m-1},\, g^{S^{m-1}} \, )$ is the standard $(m-1)$-dimensional unit sphere, and  $f:[0,\infty)\to [0,\infty) $ is a smooth function satisfying the following boundary conditions: 
\begin{equation}\notag\label{condizioni-su-f}
f(0)=0 \,, \quad f'(0)=1 \quad {\rm and}\quad f(r)>0 \quad {\rm if} \,\, r>0.
\end{equation}

Note (see \cite{GW79}) that in this model, $r$ measures the geodesic distance from the pole $o$, and the warping function $f$ relates to the the radial curvature by the following Jacobi equation:
\begin{equation}\notag\label{equazionecurvaturaradiale}
    f''(r)\,+\,K(r)\,f(r)\,=\,0 \,\, ,\,\, f(0)=0\,, \,\,f'(0)=1,
\end{equation}
where the radial curvature $K(r)$ ($r>0$) of the model means the sectional curvature of any plane containing the radial direction $\partial_r$.

We use $M^m_f(o)$ to denote a model  $(M^m(o),\,g)$ defined above, but, to abuse the notation and terminology,  we sometimes also allow $f(r)$ to define on a finite interval $[0,\,b]$, with $f(b)=0$ and $f'(b)=-1$. Sometimes, we also use the standard warped product notation like $I\times_fS^{m-1}$. In the following, we will use the following warped product models to represent a space form of constant sectional curvature $K$:
\begin{equation}\notag
\begin{aligned}
K=0:\;  &\mathbb{R}^m=M^m_f(o), \;{\rm with}\;f(r)=r,\;{\rm for}\; r\in [0, + \infty)\\
K=1:\;  &S^m\setminus\{{\rm south\;pole}\}=M^m_f(o), \;{\rm with}\\
&f(r)=\sin r,\;{\rm for}\; r\in [0,\pi];\\
K=-1:\;  &H^m=M^m_f(o), \;{\rm with}\\\notag &f(r)=\sinh r\, \;{\rm for}\; r\in [0, + \infty).
\end{aligned}
\end{equation}
Recall (see \cite{MR13}) that {\bf an equivariant map} is a map $\phi:(\mathbb{R}^+\times S^{m-1}, {\rm d}r^2+\sigma^2(r)g^{S^{m-1}})\to (\mathbb{R}^+\times S^{n-1}, {\rm d}\rho^2+\lambda^2(\rho)g^{S^{n-1}})$ between two model spaces defined by $\phi(r, \theta)=(\rho(r), \varphi(\theta))$, where $\varphi: S^{m-1}\to S^{n-1}$ is an eigenmap with eigenvalue $2e(\varphi)=2k>0$. In particular, a map  $\phi:(\mathbb{R}^+\times S^{m-1}, {\rm d}r^2+\sigma^2(r)g^{S^{m-1}})\to (\mathbb{R}^+\times S^{m-1}, {\rm d}\rho^2+\lambda^2(\rho)g^{S^{m-1}})$ with $\phi(r, \theta)=(\rho(r), \theta)$ is called a {\bf rotationally symmetric map}.

Now we are ready to study equivariant biharmonic maps. First, we have the following biharmonic equation for equivariant maps between two models.
\begin{proposition}\label{MT1}
For an eigenmap $\varphi: S^{m-1}\to S^{n-1}$ with eigenvalue $2e(\varphi)=2k$, the equivariant map $\phi:(\mathbb{R}^+\times S^{m-1}, {\rm d}r^2+\sigma^2(r)g^{S^{m-1}})\to (\mathbb{R}^+\times S^{n-1}, {\rm d}\rho^2+\lambda^2(\rho)g^{S^{n-1}})$ with $\phi(r, \theta)=(\rho(r), \varphi(\theta))$ is a biharmonic map if and only if $\rho$ is a solution of
\begin{equation}\label{Oe190}
\begin{cases}
F=\rho''+(m-1)\frac{\sigma'}{\sigma}\rho'-2k\,\frac{\lambda\lambda'(\rho)}{\sigma^2},\\
F''+(m-1)\frac{\sigma'}{\sigma}F'-2k\,\frac{(\lambda\lambda')'(\rho)}{\sigma^2}F=0,
\end{cases}
\end{equation}
or equivalently,
\begin{equation}\label{BiL}
\begin{aligned}
\Delta^2 \rho -2k \Delta\left(\frac{\lambda\lambda'(\rho)}{\sigma^2}\right) -2k\,\frac{(\lambda\lambda')'(\rho)}{\sigma^2}\left(\Delta \rho-2k\frac{\lambda\lambda'(\rho)}{\sigma^2}\right)=0.
\end{aligned}
\end{equation}
\end{proposition}
\begin{proof}
Take a system of local coordinate $\{r, u^i\}$ on $M^m_{\sigma}(o)$ (resp.  $\{\rho, y^b\}$ on $M^n_{\lambda}(o)$) so that $\{ u^i\}$ is a local coordinate system on $S^{m-1}$ (resp.  $\{ y^b\}$ is a local coordinate system on $S^{n-1}$). We adopt the following convention for the range of the indices: $A, B, C,D=0, 1, 2,\cdots, m-1;\;\;i, j, k, l=1, 2,\cdots, m-1;\;\; \alpha, \beta, \xi, \mu=0,1,2,\cdots, n$; and $a, b, c, d=1, 2,\cdots, n-1$. Let ${\bar g}={\rm d}r^2+\sigma^2(r)g^{S^{m-1}}$ and ${\tilde h}={\rm d}\rho^2+\lambda^2(\rho)g^{S^{n-1}}$, then we have
\begin{align}\notag
&{\bar g}_{00}={\bar g}^{00}=1,\; {\bar g}_{0i}={\bar g}^{0i}=0,\;\; {\bar g}_{ij}=\sigma^2(r)g_{ij},\;\; {\bar g}^{ij}=\sigma^{-2}g^{ij},\\\notag
&{\tilde h}_{00}={\tilde h}^{00}=1,\; {\tilde h}_{0a}={\tilde h}^{0a}=0,\;\; {\tilde h}_{ab}=\lambda^2(\rho)h_{ab},\;\;{\tilde h}^{ab}=\lambda^{-2}(\rho)h^{ab},
\end{align}
where $g_{ij}$ and $h_{ab}$ denote the components of the standard metrics of $S^{m-1}$ and $S^{n-1}$ respectively.
  By using the local expression $\phi(r, \theta)=(\rho(r), \varphi(\theta))$ or, $\phi^0(r, \theta)=\rho(r), \; \phi^b(r, \theta)=\varphi^b(\theta)$, we compute
\begin{equation}\notag
\begin{aligned}
\tau(\phi)=&{\bar g}^{AB}(\phi^0_{AB}-{\bar\Gamma}_{AB}^C\phi^0_C+{\tilde\Gamma}_{\alpha\beta}^0\phi^{\alpha}_A\phi^{\beta}_B)\partial_{\rho}+{\bar g}^{AB}(\phi^d_{AB}-{\bar\Gamma}_{AB}^C\phi^d_C+{\tilde\Gamma}_{\alpha\beta}^d\phi^{\alpha}_A\phi^{\beta}_B)\partial_{d}\\
=&(\rho''-{\bar g}^{ij}{\bar\Gamma}_{ij}^0\rho'+{\bar g}^{ij}{\tilde\Gamma}_{ab}^0\varphi^{a}_i\varphi^{b}_j)\partial_{\rho}+\sigma^{-2}g^{ij}(\varphi^d_{ij}-{\bar\Gamma}_{ij}^k\varphi^d_k+{\tilde\Gamma}_{ab}^d\varphi^{a}_i\varphi^{b}_j)\partial_{d}\\
=&\left(\rho''+(m-1)\frac{\sigma'}{\sigma}\rho'-\frac{2k\lambda\lambda'(\rho)}{\sigma^2}\right)\partial_{\rho}+\sigma^{-2}(0,\;\tau(\varphi)),
\end{aligned}
\end{equation}
where in obtaining the last equality we have used the identities of the connection coefficients of  the warped product metrics 
\begin{align}\notag
{\bar\Gamma}_{ij}^0=-\sigma\sigma'g_{ij},\;\;{\tilde\Gamma}_{ab}^0=-\lambda\lambda'h_{ab},
\end{align}
and the fact that the connection coefficients ${\bar\Gamma}_{ij}^k$ and ${\tilde\Gamma}_{ab}^d$ of the metrics $\sigma^2(r)g^{S^{m-1}}$ and $\lambda^2(\rho)g^{S^{n-1}}$  agree with those of the standard metrics on $S^{m-1}$ and $S^{n-1}$  respectively. Using the fact that an eigenmap is a harmonic map with $\tau(\varphi)=0 $, we conclude that the tension field of the equivariant map $\phi$ is given by $\tau(\phi)=F\partial_{\rho}$ for $F$ given in the first equation of (\ref{Oe190}).

To compute the bitension field, we have (cf. (7) in \cite{Ou09}) 
\begin{equation}\label{Oe192}
\begin{aligned}
\tau_2(\phi)=-J(\tau(\phi))=-J(F\partial_{\rho}))=-[FJ(\partial_{\rho})-(\Delta F)\partial_{\rho}-2\nabla^{\phi}_{\nabla F}\partial_{\rho}],
\end{aligned}
\end{equation}
where $J$ is the Jacobi operator defined by
\begin{equation}\label{JO}
J^{\phi}_{\bar g}(X)=-{\rm
Tr}_{\bar g}[(\nabla^{\phi}\nabla^{\phi}-\nabla^{\phi}_{\bar \nabla})X
-  {\rm R}^{N}({\rm d}\phi, X){\rm d}\phi]
\end{equation}
for any vector field $X$ along the map $\phi$

A further computation using the assumptions that $\tau(\varphi)=0$ and $|{\rm d}\varphi|^2=2k$  yields
\begin{equation}\notag
\begin{aligned}
\Delta F=&F''+(m-1)\frac{\sigma'}{\sigma}F',\;\;\;\nabla^{\phi}_{\nabla F}\partial_{\rho}=0,\; {\rm and}\\
-J(\partial_{\rho})=&{\bar g}^{00}[\nabla^{\phi}_{\partial_r}\nabla^{\phi}_{\partial_r}\partial_{\rho}-\nabla^{\phi}_{{\bar\nabla}_{\partial_r}{\partial_r}}\partial_{\rho}-{\rm \tilde R}({\rm d}\phi(\partial_r), \partial \rho){\rm d}\phi(\partial_r)]\\
&+{\bar g}^{ij}[\nabla^{\phi}_{\partial_i}\nabla^{\phi}_{\partial_j}\partial_{\rho}-\nabla^{\phi}_{{\bar\nabla}_{\partial_i}{\partial_j}}\partial_{\rho}-{\rm \tilde R}({\rm d}\phi(\partial_i), \partial \rho){\rm d}\phi(\partial_j)]\\
=&\frac{\lambda'}{\lambda\sigma^2}\tau(\varphi)-\frac{2e(\varphi)}{\sigma^2}(\lambda\lambda')'(\rho)\partial_{\rho}=-2k\frac{(\lambda\lambda')'(\rho)}{\sigma^2}\partial_{\rho}.
\end{aligned}
\end{equation}
Using these and (\ref{Oe192}) we have
\begin{equation}\notag
\tau_2(\phi)=\left(F''+(m-1)\frac{\sigma'}{\sigma}F'-2k\frac{(\lambda\lambda')'(\rho)}{\sigma^2}F \right)\partial_{\rho},
\end{equation}
which gives Equation (\ref{Oe190}). Equation (\ref{BiL}) follows from (\ref{Oe190}) and the fact that $\Delta \alpha=\alpha''+(m-1)\frac{\sigma'}{\sigma}\alpha'$ for a function $\alpha=\alpha(r)$ on $(\mathbb{R}^+\times S^{m-1}, {\rm d}r^2+\sigma^2(r)g^{S^{m-1}})$.
\end{proof}

It is easily seen that Proposition \ref{MT1} gives a generalization of the following 
\begin{proposition}\label{on04}{\rm \cite{BMO07a}}
For an eigenmap $\varphi: S^{m-1}\to S^{n-1}$ with eigenvalue $2e(\varphi)=2k>0$, the equivariant map $\phi:(\mathbb{R}^+\times S^{m-1}, {\rm d}r^2+r^2g^{S^{m-1}})\to (\mathbb{R}^+\times S^{n-1}, {\rm d}\rho^2+\lambda^2(\rho)g^{S^{n-1}})$ from Euclidean space with $\phi(r, \theta)=(\rho(r), \varphi(\theta))$ is a biharmonic map if and only if $\rho$ is a solution of
\begin{equation}\label{Oe193}
\begin{cases}
F=\rho''+(m-1)\frac{1}{r}\rho'-2k\,\frac{\lambda\lambda'(\rho)}{r^2},\\
F''+(m-1)\frac{1}{r}F'-2k\,\frac{(\lambda\lambda')'(\rho)}{r^2}F=0.
\end{cases}
\end{equation}
\end{proposition}

Proposition \ref{MT1}  also includes the following corollary as a special case.

\begin{corollary}\label{Ra01}{\rm (see \cite{WOY14} for the case of $m=2$ and  \cite{MOR15} for general $m$)}
The rotationally symmetric map $\phi:(\mathbb{R}^+\times S^{m-1}, {\rm d}r^2+\sigma^2(r)g^{S^{m-1}})\to (\mathbb{R}^+\times S^{m-1}, {\rm d}\rho^2+\lambda^2(\rho)g^{S^{m-1}})$  with $\phi(r, \theta)=(\rho(r), \theta)$ is a biharmonic map if and only if $\rho$ is a solution of
\begin{equation}\label{Rae00}
\begin{cases}
F=\rho''+(m-1)\frac{\sigma'}{\sigma}\rho'-(m-1)\,\frac{\lambda\lambda'(\rho)}{\sigma^2},\\
F''+(m-1)\frac{\sigma'}{\sigma}F'-(m-1)\,\frac{(\lambda\lambda')'(\rho)}{\sigma^2}F=0.
\end{cases}
\end{equation}
\end{corollary}

\begin{example}
One can check (cf. e.g., Example 9.10 in \cite{OC20}) that the Hopf fibration $\varphi:S^3\to S^2$ is an eigenmap with $2e(\varphi)=2k=|{\rm d}\varphi|^2=8$, so the equivariant map $\phi: S^4=([0,\pi]\times S^3, {\rm d}r^2+\sin^2 r g^{S^3})\to S^3=([0,\pi]\times S^2, {\rm d}\rho^2+\sin^2 \rho g^{S^2})$ with $\phi(r,\theta)=(\rho(r), \varphi(\theta))$ is a biharmonic map if and only if $\rho=\rho(r)$ solves the equation
\begin{equation}\notag\label{Oe195}
\begin{cases}
F=\rho''+3(\cot r)\rho'-4\,\frac{\sin 2\rho}{\sin^2r},\\
F''+3(\cot r)F'-8\,\frac{\cos 2\rho}{\sin^2r}F=0.
\end{cases}
\end{equation}
\end{example}
The following classification of rotationally symmetric conformal biharmonic maps between space forms was obtained in \cite{MOR15}.
\begin{theorem}\label{Ra02}{\rm \cite{MOR15}} For $m>2$, a rotationally symmetric conformal map $\phi:(\mathbb{R}^+\times S^{m-1}, {\rm d}r^2+\sigma^2(r)g^{S^{m-1}})\to (\mathbb{R}^+\times S^{m-1}, {\rm d}\rho^2+\lambda^2(\rho)g^{S^{m-1}})$, $\phi(r, \theta)=(\rho(r), \theta)$, between model spaces of constant sectional curvature is proper biharmonic if and only if, up to a homothety of the domain or the target space, it is one of the following maps:\\
(i) The inversion in $S^3$,  $\phi:\mathbb{R}^4\setminus\{0\}\equiv(\mathbb{R}^+\times S^{3}, {\rm d}r^2+r^2g^{S^{3}})\to \mathbb{R}^4\setminus\{0\}\equiv(\mathbb{R}^+\times S^{3}, {\rm d}\rho^2+\rho^2g^{S^{3}})$  with $\phi(r, \theta)=(1/r, \theta)$, which can be described, in Cartesian coordinates, as  $\phi:\mathbb{R}^4\setminus\{0\}\to \mathbb{R}^4\setminus\{0\}$, $\phi(x)=x/|x|^2$.\\
(ii) The inverse of the stereographic projection $\phi:\mathbb{R}^4\equiv(\mathbb{R}^+\times S^{3}, {\rm d}r^2+r^2g^{S^{3}})\to S^m\setminus\{{\rm south\;pole}\}\equiv(\mathbb{R}^+\times S^{3}, {\rm d}\rho^2+\sin^2 \rho g^{S^{3}})$  with $\phi(r, \theta)=(2\tan^{-1} r, \theta)$.\\
(iii) $\phi:B^4\equiv(\mathbb{R}^+\times S^{3}, {\rm d}r^2+r^2g^{S^{3}})\to H^4\equiv(\mathbb{R}^+\times S^{3}, {\rm d}\rho^2+ \sinh^2 \rho g^{S^{3}})$  with $\phi(r, \theta)=(2 \tanh^{-1} r, \theta)$ for  $r\in [0, 1)$. So, the map is from an open unit ball  into a hyperbolic space form of constant sectional curvature $K=-1$.
\end{theorem}

In dimension $4$, we have the following generalization which gives a complete classification of rotationally symmetric biharmonic conformal map from a space form into a model.

\begin{theorem}\label{O191}A rotationally symmetric map 
\begin{equation}\label{Ro4d}
\begin{aligned}
    \phi: M^4(c) \to M _{\lambda}^4 (o)=&(\mathbb{R}^+\times S^3, {\rm d}\rho^2+\lambda^2(\rho)g^{S^3}),\\
    \phi(r,\,\theta) =& (\rho(r),\,\theta)
  \end{aligned}
\end{equation}
from a $4$-dimensional space form into a model is a  proper biharmonic conformal diffeomorphism if and only the target model is a space form, and  up to a homothety, the map $\phi$ is one of the three maps given in Theorem \ref{Ra02}. 
\end{theorem}
\begin{proof}
Since a space form is an Einstein manifold,  we can use Theorem 2.3 in \cite{BO18} (see also Theorem 11.13 in \cite{OC20}) to conclude that the rotationally symmetric conformal map $\phi$ given in (\ref{Ro4d}) is  biharmonic if and only if its conformal factor $\rho'$ solves the equation
\begin{equation}\notag
\Delta \rho'-3c\,\rho'=A \rho'^3,
\end{equation}
where $A$ is a constant relating to the scalar curvature ${\rm Scal}^N$ of the target manifold and the sectional curvature $c$ of the space form $M^4(c)$ by
 \begin{equation}\label{Oe196}
 6A+\frac{6c}{\rho'^2}+{\rm Scal}^N=0.
 \end{equation}

A straightforward computation shows that by performing the following change of variable 
\begin{align}\notag
\begin{cases}
r=e^t,\;\hskip1.7cm {\rm for}\; c=0;\\
r=2 \tan^{-1}\,e^t,\;\;\;\; {\rm for}\; c>0;\\
r=2 \tanh^{-1}\,e^t,\;\; {\rm for}\; c<0,
\end{cases}
\end{align}
the conformality condition of the rotationally symmetric map reads
\begin{equation}\notag
\rho'(t)=\lambda (t).
\end{equation}
Substituting this into (\ref{Oe196}) we have
 \begin{equation}\label{Oe197}
 6A+\frac{6c}{\lambda^2}+{\rm Scal}^N=0.
 \end{equation}
 On the other hand, by a straightforward computation of the scalar curvature of the warped product manifold, we have
  \begin{equation}\label{Oe198}
{\rm Scal}^N= \frac{6}{\lambda^2}-6\frac{\lambda\lambda''}{\lambda^2}-6\frac{\lambda'^2}{\lambda^2}.
 \end{equation}
Combining (\ref{Oe197}) and (\ref{Oe198}) we obtain
  \begin{equation}\notag\label{Oe199}
(\lambda^2)''-2A\lambda^2= 2(1+c).
 \end{equation}
 Solving this ODE yields
 \begin{align}\notag
 \lambda^2(\rho)=
\begin{cases}
(1+c)\rho^2+C_1\rho+C_2, \hskip4cm {\rm for}\; A=0;\\
C_1e^{\sqrt{2A}\,\rho}+C_2e^{-\sqrt{2A}\,\rho}+C_3,\;\hskip1.7cm {\rm for}\; A>0;\\
C_1\cos\sqrt{2|A|}\,\rho+C_2\sin\sqrt{2|A|}\,\rho+C_3,\;\;\;\; {\rm for}\; A<0.
\end{cases}
\end{align}
One can easily check that by using the boundary condition $\lambda(0)=0$, we have exactly
 \begin{align}\notag
 \lambda(\rho)=
\begin{cases}
\sqrt{1+c}\,\rho, \hskip2.2cm {\rm for}\; A=0;\\
C\sinh(\frac{\sqrt{2A}}{2}\rho),\;\;{\rm for}\; A>0\; {\rm and\; some\;constant}\; C>0;\\
C\sin(\frac{\sqrt{2|A|}}{2}\rho),\,\;\;\;\; {\rm for}\; A<0\; {\rm and\; some\;constant}\; C>0,
\end{cases}
\end{align}
which means the target manifold has constant sectional curvature. Using this and Theorem \ref{Ra02} we obtain the theorem.
\end{proof}

Note  that without the conformality requirement on the maps, the biharmonic equations for equivariant (even rotationally symmetric) maps between general models are still very difficult to solve although they are ordinary differential equations (see \cite{WOY14} for the study of rotationally symmetric biharmonic maps in the simplest case $\varphi:(\mathbb{R}^+\times S^1, {\rm d}r^2+\sigma^2(r){\rm d}\theta^2)\to (\mathbb{R}^+\times S^{1}, {\rm d}\rho^2+\lambda^2(\rho){\rm d}\phi^2)$  with $\varphi(r, \theta)=(\rho(r), \theta)$). However, we can have complete solutions for the maps between Euclidean domains as follows.
\begin{proposition}\label{on06}{\rm \cite{BMO07a}}
For an eigenmap $\varphi: S^{m-1}\to S^{n-1}$ with eigenvalue $2e(\varphi)=2k>0$, the equivariant map $\phi:(\mathbb{R}^+\times S^{m-1}, {\rm d}r^2+r^2g^{S^{m-1}})\to (\mathbb{R}^+\times S^{n-1}, {\rm d}\rho^2+\rho^2g^{S^{n-1}})$, $\phi(r, \theta)=(\rho(r), \varphi(\theta))$ between Euclidean domains  is harmonic if and only if
\begin{equation}\label{Oe1915}
\rho(r)=c_1r^{k_1}+c_2r^{k_2} ,\;\;{\rm for}\; k_{1,2}=\frac{-(m-2)\pm\sqrt{(m-2)^2+8k}}{2},
\end{equation}
where $c_1, c_2$ are arbitrary such that $\rho$ takes values in $(0,\infty)$. The equivariant map is biharmonbic if and only if
\begin{equation}\label{Oe1910}
\rho(r)=\begin{cases}
     c_1r^3+c_2\,r\ln r+c_3\,r+c_4\,r^{-1}, \quad {\rm when}\;
    m=2\,\; {\rm and}\,\; k=\frac{1}{2}, {\rm or}
    \\   
   \frac{c_1}{2(m+2k_1)}r^{k_1+2}+\frac{c_2}{2(m+2k_2)}r^{k_2+2}
   +c_3r^{k_1}+c_4r^{k_2} ,
\quad {\rm otherwise},
  \end{cases}
\end{equation}
for some constants $c_1, c_2, c_3, c_4$ and $k_{1}, k_2$ given by (\ref{Oe1915}).
\end{proposition}

\begin{remark}
Note that for $m=2$, beside the solutions given in the first family of (\ref{Oe1910}), the second family also provides many solutions. For example,  for the eigenmap $\varphi: S^1\to S^1$, $\varphi(z)=z^2,$ with eigenvalue $-2k=-4$, we have biharmonic map $\phi:\mathbb{R}^2\setminus\{0\}\to \mathbb{R}^2\setminus\{0\}$ given by $\phi(r, \theta)=(\rho(r), 2\theta)$ with $\rho(r)=\frac{c_1}{20}r^6-\frac{c_2}{12}r^{-2}+c_3r^4+c_4r^{-4}$.
\end{remark}

Note that for biharmonic functions on a star-shaped Euclidean domain, we have the well-known Almansi property (\cite{Al99}) which states that  any real-valued biharmonic function $u:\mathbb{R}^m\supset \Omega\to \mathbb{R}$ on a star-shaped domain centered at the origin can be expressed as
\begin{align}\notag\label{khe}
u(x)= h_1(x)+|x|^{2} h_2(x), \;\forall\; x\in \Omega,
\end{align}
where $h_i:\mathbb{R}^m\supset \Omega\to \mathbb{R}$ are harmonic functions.

One can easily check (see also \cite{OC20}, Corollary 10.3)  that this Amansi property generalizes to biharmonic maps from a star-shaped Euclidean domain into another Euclidean space. The following corollary shows that  when the maps are equivariant, the condition of the domain being star-shaped can be dropped.
\begin{corollary}\label{O191}
For $m> 2$, an equivariant biharmonic map  $\phi:\mathbb{R}^m\setminus\{0\}\to \mathbb{R}^n\setminus\{0\}$ given by $\phi(r, \theta)=(\rho(r), \varphi(\theta))$ has Almansi property, i.e., 
\begin{equation}
\phi(r, \theta)=r^2\phi_1(r, \theta)+\phi_2(r, \theta),
\end{equation}
 where $\phi_1, \phi_2:\mathbb{R}^m\setminus\{0\}\to \mathbb{R}^n\setminus\{0\}$ are two equivariant harmonic maps defined by the same eigenmap $\varphi:S^{m-1}\to S^{n-1}$.
\end{corollary}
\begin{proof}
By Proposition \ref{on06}, for $m>2$, any equivariant biharmonic map $\phi: \mathbb{R}^m\setminus\{0\}\to \mathbb{R}^n\setminus\{0\}$ can be described, using geodesic polar coordinates, as $\phi(r, \theta)=(\rho(r), \varphi(\theta))$ for some eigenmap $\varphi$ and $\rho$ given by
\begin{align}\notag
\rho(r)= r^2(c_1r^{k_1}+c_2r^{k_2}) +c_3r^{k_1}+c_4r^{k_2} =r^2\rho_1(r)+\rho_2(r),
\end{align}
where $ \rho_1(r)=c_1r^{k_1}+c_2r^{k_2}$ and $\rho_2(r)=c_3r^{k_1}+c_4r^{k_2}$. It also follows from Proposition \ref{on06} that the equivariant maps $\phi_1, \phi_2:\mathbb{R}^m\setminus\{0\}\to \mathbb{R}^n\setminus\{0\}$ with $\phi_1(r, \theta)=(\rho_1(r),\varphi(\theta))$ and $\phi_2(r, \theta)=(\rho_2(r),\varphi(\theta))$ are harmonic maps. Now a straightforward computation shows that $\phi(r, \theta)=r^2\phi_1(r, \theta)+\phi_2(r, \theta)$, from which the corollary follows.
\end{proof}
\begin{remark}
We would like to point out  that $\mathbb{R}^m\setminus\{0\}$ is not a star-shaped region, but the Almansi property still holds for an equivariant  biharmonic map due to the symmetry of the map and a large enough domain ($m\ge 3$); Note also that for $m=2$,  we have equivariant biharmonic maps $\phi:\mathbb{R}^2\setminus\{0\}\to \mathbb{R}^2\setminus\{0\}$ with $\phi(r, \theta)=(\rho(r), \theta)$ with $\rho(r)=r\ln r$, which does not have the Almansi property.
\end{remark}

Note that we can use Proposition \ref{on06}  with $m=4$ and a boundary condition to have the following corollary which recovers Theorem 4.12 in \cite{Ba08} obtained in a different way.
\begin{corollary}\label{O191}{\rm\cite{Ba08}}
A rotationally symmetric biharmonic map $\phi:B_R^4(0)\to \mathbb{R}^4$, $\phi(r, \theta)=(\rho(r), \theta)$ with $\rho(0)=0$ from an open  ball of radious $R$ is the restriction of  the map $\phi(x)=C_1 x+C_2|x|^2x$. If a pole at the origin is allowed, then, the map $\phi(x)=C_1x+C_2|x|^2x+C_3\frac{x}{|x|^2}+C_4\frac{x}{|x|^4}$ is a biharmonic map. In particular, the inversion in $3$-sphere $\phi(x)=\frac{x}{|x|^2}$ is a conformal proper biharmonic map which is also a biharmonic morphism.
\end{corollary}
\begin{proof}
Using Proposition \ref{on06}  with $m=4$ and $2k=3$ we have $k_1=-3, k_2=1$ from which we obtain the general solution $\rho(r)=C_1r+C_2r^3+C_3r^{-1}+C_4r^{-3}$. It is easy to check that with this  solution of $\rho(r)$ the rotationally symmetric map $\phi(r, \theta)=(\rho(r), \theta)$ corresponds to the map $\phi(x)=C_1x+C_2|x|^2x+C_3\frac{x}{|x|^2}+C_4\frac{x}{|x|^4}$ in Cartesian coordinates. Now, if we add the boundary condition $\rho(0)=0$, then we have the solution $\rho(r)=C_1r+C_2r^3$ which corresponds to the map $\phi(x)=C_1 x+C_2|x|^2x$.
\end{proof}

\section{Second variation formula and a classification of stable proper biharmonic maps into a positively curved space form}
The following second variation formula for biharmonic maps was derived by Jiang.
\begin{theorem} \cite{Jiang86}
For a biharmonic map $\phi:(M^m,g)\to (N^n,h)$ from a compact Riemannian manifold, and a variation $\phi_t$ of $\phi$ with variation vector field $V$, we have the following second variation formula
\begin{equation}\label{2Vformula}
\begin{aligned}
\frac{{\rm d}^2 }{{\rm d} t^2}E_2(\phi_t)|_{t=0}=&\int_M [|J^{\phi}(V)|^2+ {\rm R}^N(V, \tau(\phi), V,\tau(\phi) ] {\rm d} v_g\\
&\hskip-.5in -\sum_{i=1}^m\int_M \langle V, (\nabla^{N}_{{\rm d}\phi(e_i)}{\rm R}^N)({\rm d}\phi(e_i), \tau(\phi))V \\
&\hskip-.5in + (\nabla^{N}_{\tau(\phi)}{\rm R}^N)({\rm d}\phi(e_i), V){\rm d}\phi(e_i)\\
&\hskip-.5in +2{\rm R}^N({\rm d}\phi(e_i), V)\nabla^{\phi}_{e_i}\tau(\phi)+2{\rm R}^N({\rm d}\phi(e_i), \tau(\phi))\nabla^{\phi}_{e_i}V \rangle {\rm d} v_g,
\end{aligned}
\end{equation}
where $J^{\phi}$ is the Jacobi operator of $\phi$ defined by (\ref{JO}) and $\{e_i\}$ is an orthonormal frame on $M$.
\end{theorem}

First, we give the following form of the second variation formula for biharmonic maps, which is very convenient to use and will be used to prove a classification theorem for stable biharmonic maps form a compact manifold into a space form of positive sectional curvature.

\begin{corollary} 
For a biharmonic map $\phi:(M^m,g)\to (N^n,h)$ from a compact Riemannian manifold into a space form of positive curvature  $c$, and a variation $\phi_t$ of $\phi$ with variation vector field $V$, we have the following second variation formula
\begin{equation}\label{2VfO}
\begin{aligned}
\frac{{\rm d}^2 }{{\rm d} t^2}E_2(\phi_t)|_{t=0}=&\int_M \big\{|J^{\phi}(V)|^2-c\big[|V|^2|\tau(\phi)|^2+\langle V, \tau(\phi)\rangle^2\\
&-2|V|^2\,{\rm div}\,\langle {\rm d}\phi, \tau(\phi)\rangle-2\langle V, \tau(\phi)\rangle\, {\rm Tr}\,\langle {\rm d} \phi, \nabla^{\phi}V\rangle\\
&+2\langle V,{\rm d}\phi({\rm grad} \langle V, \tau(\phi)\rangle) \rangle\big]\big\}{\rm d}v_g.
\end{aligned}
\end{equation}
\end{corollary}
\begin{proof}
By using (\ref{2Vformula}) and $\nabla^N{\rm R}^N=0$ for a space form, we have 
\begin{equation}\label{2Vfo1}
\begin{aligned}
&\frac{{\rm d}^2 }{{\rm d} t^2}E_2(\phi_t)|_{t=0}=\int_M [|J^{\phi}(V)|^2+ {\rm R}^N(V, \tau(\phi), V,\tau(\phi) ] {\rm d} v_g\\
&\hskip.2in -2\int_M \Big\{\sum_{i=1}^m\langle V, {\rm R}^N({\rm d}\phi(e_i), V)\nabla^{\phi}_{e_i}\tau(\phi)\\&\hskip.3in  +{\rm R}^N({\rm d}\phi(e_i), \tau(\phi))\nabla^{\phi}_{e_i}V\rangle \Big\} {\rm d}v_g.
\end{aligned}
\end{equation}
By using the curvature property of a space form and  further computations, we obtain
\begin{equation}\notag 
{\rm R}^N(V, \tau(\phi), V,\tau(\phi)=c(|V|^2|\tau(\phi)|^2-\langle V, \tau(\phi)\rangle^2)
\end{equation}
and 
\begin{equation}\notag 
\begin{aligned}
&\sum_{i=1}^m\langle V, {\rm R}^N({\rm d}\phi(e_i), V)\nabla^{\phi}_{e_i}\tau(\phi)+{\rm R}^N({\rm d}\phi(e_i), \tau(\phi))\nabla^{\phi}_{e_i}V\rangle\\
=&\; c\sum_{i=1}^m\big(\langle V, {\rm d}\phi(e_i)\rangle(e_i\langle V, \tau(\phi)\rangle)-|V|^2\langle {\rm d}\phi(e_i), \nabla^{\phi}_{e_i}\tau(\phi) \rangle\\
&-\langle V, \tau(\phi)\rangle\langle {\rm d}\phi(e_i), \nabla^{\phi}_{e_i}V \rangle \big)\\
=&\; c\big[ \langle V, {\rm d}\phi({\rm grad} \langle V,\tau(\phi)\rangle)\rangle-|V|^2\,({\rm div}\langle {\rm d}\phi, \tau(\phi)\rangle-|\tau(\phi)|^2)\\
&-\langle V, \tau(\phi)\rangle\,{\rm Tr}\langle {\rm d}\phi, \nabla^{\phi}V\rangle\big],
\end{aligned}
\end{equation}
where in obtaining the last equality we have used the identity
\begin{align}\label{je8629}
\sum_{i=1}^m\langle {\rm d}\phi(e_i), \nabla^\phi_{e_i}\tau(\phi)\rangle={\rm div}\langle {\rm d}\phi, \tau(\phi)\rangle-|\tau(\phi)|^2
\end{align}
which can be checked by a straightforward computation (see also Equation (12.6) in \cite{OC20}).

Substituting these into (\ref{2Vfo1}) we obtain the corollary.
\end{proof}

Recall that a biharmonic map $\phi:(M^m,g)\to (N^n,h)$ from a compact manifold is said to be a {\bf  stable } if the second variation of bienergy is nonnegative for every variation $\phi_t$ of $\phi$. This is equivalent to the integral in (\ref{2Vformula}) is nonnegative for any vector field $V$ along the map $\phi$.

It is clear from the definition and the second variational formula (\ref{2Vformula}) that a harmonic map from a compact manifold as a trivial biharmonic map is stable since $ \frac{{\rm d}^2 }{{\rm d} t^2}E_2(\phi_t)|_{t=0}=\int_M |J^{\phi}(V)|^2 {\rm d} v_g\ge0$ for any vector field $V$ along $\phi$.

Th following classifications of stable biharmonic maps were  proved in \cite{Jiang86}.
\begin{theorem} \cite{Jiang86}\label{JiangStable}
If a stable biharmonic map $\phi:(M^m,g)\to (N^n,h)$ from a compact Riemannian manifold into a space form of positive sectional curvature $c>0$ satisfies the first conservation law, i.e., ${\rm div}S(\phi)=0$, then it is harmonic map.
\end{theorem}

\begin{theorem}\cite{Jiang86}\label{JiangStable}
A stable biharmonic map $\phi: (M^m, g) \to \mathbb{C}P^n$   from a compact manifold into a complex projective space with constant holomorphic sectional curvature $c>0$ satisfying the first conservation law is a harmonic map provided one of the following holds:
\noindent
{\rm (i)}   $|\tau(\phi)|^2> 3\sqrt{2\,e(\phi)}\,|\nabla^{\phi}\tau(\phi)|$ at any point in $M$; or
\noindent
{\rm (ii)}  $|\tau(\phi)|^2={\rm constant}$ and $|\tau(\phi)|> 6\sqrt{c}\, e(\phi)$.
\end{theorem}

For biharmonic maps into real space form, we can prove the following theorem which replaces the assumption that $\phi$ satisfies the first conservation law in Theorem \ref{JiangStable} by requiring $|\tau(\phi)|^2$ be constant.

\begin{theorem}\label{MT2}
There exists no stable proper biharmonic map $\phi:(M^m,g)\to (N^n, h)$ from a compact Riemannian manifold without boundary into a space form of positive sectional curvature with $|\tau(\phi)|^2={\rm constant}$.
\end{theorem}
\begin{proof}
Applying the second variation formula (\ref{2VfO}) with $V=\tau(\phi)$ and a further computation, we have
\begin{equation}\label{o2Vst}
\begin{aligned}
0\; &\le \frac{{\rm d}^2 }{{\rm d} t^2}E_2(\phi_t)|_{t=0}\\&=\int_M \big\{|J^{\phi}(\tau(\phi))|^2-c\big[|\tau(\phi)|^2|\tau(\phi)|^2+\langle \tau(\phi), \tau(\phi)\rangle^2\\
&\;\; -2|\tau(\phi)|^2\,{\rm div}\,\langle {\rm d}\phi, \tau(\phi)\rangle-2\langle \tau(\phi), \tau(\phi)\rangle\, {\rm Tr}\,\langle {\rm d} \phi, \nabla^{\phi}\tau(\phi)\rangle\\
&\;\; +2\langle \tau(\phi),{\rm d}\phi(\nabla |\tau(\phi)|^2) \rangle\big]\big\}{\rm d}v_g\\
&=-4c\int_M\left( |\tau(\phi)|^4-|\tau(\phi)|^2\,{\rm div}\,\langle {\rm d}\phi, \tau(\phi)\rangle\right){\rm d}v_g\\
&=-4c\, |\tau(\phi)|^4\,{\rm Vol}(M)\le 0, \;\;{\rm since\;c>0},
\end{aligned}
\end{equation}
where the second equality was obtained by using the assumption that $|\tau(\phi)|^2$ is constant and  Equation (\ref{je8629}) whilst the third equality holds by the assumption that $|\tau(\phi)|^2$ is constant and the divergence theorem. It follows from (\ref{o2Vst}) that $\tau(\phi)\equiv 0$, and hence $\phi$ is a harmonic map.
\end{proof}
\begin{remark}
(i) As  an isometric immersion  always satisfies the first conservation law, so it follows from Theorem \ref{JiangStable} that  there exists no stable proper biharmonic compact submanifold in a Euclidean sphere $S^n$.\\
\noindent (ii) One can check that the proper biharmonic maps $S^{2n-1}\to S^n\to S^{n+1}\;(n=2, 4, 8)$ obtained from the composition of Hopf fibration followed by the inclusion ${\bf i}: S^n\to S^{n+1}$ with ${\bf i}(x)=(x, 1)/\sqrt{2}$ satisfies $|\tau(\phi)|^2={\rm constant}$, so by Theorem \ref{MT2}, it is unstable. This was also confirmed by Theorem 6.1 in \cite{LO07} which states that the composition of a  non-constant eigenmap $\varphi:(M^m, g)\to S^n(\frac{1}{\sqrt{2}})$ from a compact manifold followed by the inclusion $i:S^n(\frac{1}{\sqrt{2}})\to S^{n+1}$ is an unstable proper biharmonic map.\\
\noindent(iii) As we know  (see \cite{Le82})  that any stable harmonic map $\phi: (M^m,g)\to S^n$ from a compact Riemannian manifold into a Euclidean sphere is constant. Based on the results given in Theorems \ref{JiangStable} and  \ref{MT2}, it would be interesting to know whether there exists a stable proper biharmonic map $\phi: (M^m,g)\to S^n$ from a compact manifold into a Euclidean sphere.
\end{remark}

Note that  for biharmonic maps into a sphere we also have the following form of the second variation formula.
\begin{theorem} {\rm \cite{Oni02b}}
For a biharmonic map $\phi: (M^m, g) \to S^n$ from a compact manifold into a sphere, the Hessian of the bienergy at $\phi$ is
\begin{equation}\notag
H(E_2)_{\phi}(V,W)=\int_M\langle I^{\phi}(V), W\rangle{\rm d}v_g,
\end{equation}
where
\begin{equation}\notag
\begin{aligned}
I^{\phi}(V)=&\Delta^{\phi}(\Delta^{\phi} V)-\Delta^{\phi}({\rm Tr}\langle V, {\rm d}\phi (\cdot)\rangle {\rm d}\phi (\cdot)-|{\rm d}\phi|^2V)+ 2\langle {\rm d}\tau(\phi), {\rm d}\phi(\cdot) \rangle V\\
&+|\tau(\phi)|^2V-2 {\rm Tr}\langle V, {\rm d}\tau(\phi) (\cdot)\rangle {\rm d}\phi (\cdot) -2 {\rm Tr}\langle \tau(\phi) , {\rm d}V (\cdot)\rangle {\rm d}\phi (\cdot)\\
&-\langle\tau(\phi), V\rangle \tau(\phi)+2\langle {\rm d}V, {\rm d}\phi\rangle \tau(\phi)-{\rm Tr}\langle  {\rm d}\phi (\cdot) , \Delta^{\phi} V\rangle {\rm d}\phi( \cdot )\\
&+{\rm Tr}\langle {\rm d}\phi (\cdot),  {\rm Tr}\langle V , {\rm d}\phi (\cdot)\rangle {\rm d}\phi (\cdot)\rangle {\rm d}\phi (\cdot)-2|{\rm d}\phi|^2{\rm Tr}\langle {\rm d}\phi (\cdot) , V\rangle {\rm d}\phi (\cdot)\\
&+|{\rm d}\phi|^2\Delta^{\phi} V+|{\rm d}\phi|^4V
\end{aligned}
\end{equation}
with $\Delta^{\phi} V=\sum_{i=1}^m(\nabla^{\phi}_{e_i}\nabla^{\phi}_{e_i}-\nabla^{\phi}_{\nabla^M_{e_i}{e_i}})V$ denoting the Rough Laplacian along the map $\phi$.
\end{theorem}
For some further study of the second variation of bienergy and the stability and indices of  biharmonic maps see \cite{Jiang86, Oni02b, LO07, MOR19} and the references therein.


\begin{thebibliography}{199}
\bibitem{Al99} E. Almansi, {\em Sull'integrazione dell'equazione differenziale $\Delta^{2n}=0$},  Annali di Mat., 2(1899), 1--51.

\bibitem{Ba08} P. Baird, {\em  Stress-energy tensors and the Lichnerowicz Laplacian}, J. Geom. Phys.,  58 (2008), 1329-1342.

\bibitem{BO18}  P. Baird and  Y. -L. Ou, {\em Biharmonic conformal maps in dimension four and equations of Yamabe-type},  J. Geom. Anal. 28(4) (2018),  3892-3905.

\bibitem{BMO07a} A.  Balmu\c s, S. Montaldo, C. and Oniciuc, {\em Biharmonic maps between warped product manifolds},  J. Geom. Phys., 57. (2007), 449--466.


\bibitem{Ch17}  B. -Y. Chen, {\em Differential geometry of warped product manifolds and submanifolds}, World Scientific, Hackensack, N. J., 2017. 

\bibitem{GW79} R. E. Greene and H. Wu, {\em Function theory on manifolds which possess a pole}, Lecture Notes in Mathematics, 699. Springer, Berlin, 1979. 

\bibitem{Jiang86} G. Y. Jiang, {\em $2$-Harmonic maps and their first and second variational formulas}, Chin. Ann. Math. Ser. A,  7(1986), 389-402.
\bibitem{Le82} P. -F. Leung,  {\em On the stability of harmonic maps}, in  Lecture Notes in Math. 949, Springer Verlag 1982, 122--129.
\bibitem{MR13} S. Montaldo and A. Ratto,  {\em A general approach to equivariant biharmonic maps},  Mediterr. J. Math., 10 (2013), 1127--1139.
\bibitem{MOR15} S. Montaldo, C.  Oniciuc,  and A. Ratto, {\em Rotationally symmetric biharmonic maps between models},  J. Math. Anal. Appl.,  431 (2015), 4494--505.
\bibitem{MOR19}  S. Montaldo, C.  Oniciuc,  and A. Ratto, {\em Index and nullity of proper biharmonic maps in spheres}, preprint, 2019
\bibitem{LO07}  E. Loubeau and  C. Oniciuc, {\em On the biharmonic and harmonic indices of the Hopf map},  Trans. Amer. Math. Soc., 359(11) (2007),  5239--5256.
\bibitem{Ou09} Y. -L. Ou, {\em On conformal biharmonic immersions}, Ann. Global Anal. Geom. 36, 2009, 133-142.
\bibitem{OC20} Y. -L. Ou and B. -Y. Chen, {\em Biharmonic submanifolds and biharmonic maps in Riemasnnian geometry}, World Scientific, Hackensack, N. J., 2020. 
\bibitem{Oni02b} C. Oniciuc, {\em On the second variation formula for biharmonic maps to a sphere},  Publ. Math. Debrecen,  61 (2002), 613--622.
\bibitem{WOY14} Z. -P. Wang, Y. -L. Ou and H. -C. Yang, {\em Biharmonic maps from a 2-sphere},  J. Geom. Phys.,  77 (2014), 86--96.

\end{thebibliography}
\end{document}